\documentclass[12pt]{amsart}
\usepackage{cite}
\usepackage{graphicx}
\usepackage{amsmath, amsfonts, amssymb,latexsym, graphics, graphicx, mathrsfs, manfnt, textcomp, appendix, chemarrow}
\usepackage{hyperref}
\usepackage[all]{xy}

\author{Elif Uyan\i k}
\address{Elif Uyan\i k, Department of Mathematics \\ Middle East Technical University \\ 06800 Ankara Turkey}
\email{euyanik@metu.edu.tr}
\author{Murat H. Yurdakul}
\address{Murat H. Yurdakul, Department of Mathematics \\ Middle East Technical University \\ 06800 Ankara Turkey}
\email{myur@metu.edu.tr}

\title[Cauchy product operator]{A note on triangular operators on Smooth Sequence Spaces}

\thanks{This research was partially supported by the Turkish Scientific and Technological Research Council}

\numberwithin{equation}{section}

\theoremstyle{thmit} 
\newtheorem{theorem}{Theorem}[section]
\newtheorem{remark}{Remark}[section]

\newtheorem{lemma}[theorem]{Lemma}

\begin{document}

\maketitle

\begin{abstract}
	For a scalar sequence ${(\theta_n)}_{n \in \mathbb{N}}$, let $C$ be the matrix defined by $c_n^k = \theta_{n-k+1}$ if $n \geq k$, $c_n^k = 0$ if $n < k.$ The map between K\"{o}the spaces $\lambda(A)$ and $\lambda(B)$ is called a Cauchy Product map if it is determined by the triangular matrix $C$. In this note we introduced some necessary and sufficient conditions for a Cauchy Product map on a nuclear K\"{o}the space $\lambda(A)$ to nuclear $G_1$-space $\lambda(B)$ to be linear and continuous. Its transpose is also considered. 
\end{abstract}
	
\dedicatory{Dedicated to the memory of Prof. Dr. Tosun Terzio{\~g}lu}

\section{Introduction}
We refer the reader to \cite{Mei97}, \cite{Pie72} and \cite{Ram79} for the terminology used but not defined here. Let $A=(a_n^k)_{n,k\in\mathbb{N}}$ be a matrix of real numbers such that $0\leq a_n^k \leq a_n^{k+1}$ for all $n,k$ and $\displaystyle \sup_{k} a_{n}^{k} > 0$. The $\ell^{1}$ - K\"{o}the space $\lambda(A)$ defined by the matrix $A$ is the space of all sequences of scalars $x = (x_n)$ such that
$$
\|x\|_k = \sum_n |x_n| a_n^k < \infty, \quad \forall k \in \mathbb{N}\
$$ 
With the topology generated by the system of seminorms $\left\{\left\|.\right\|_k, k \in \mathbb{N}\right\}$, it is a Fr{\'e}chet space.

    The topological dual of $\lambda(A)$ is isomorphic to the space of all sequences $u$ for which $\left|u_{n}\right| \leq C a_{n}^{k}$ for some $k$ and $C > 0.$
		
    It is well known that a K\"{o}the space $\lambda(A)$ associated with the matrix $A$ is nuclear if and only if for each $k$ there exists $m$ such that
$$
\sum_n \frac{a_n^k}{a_n^m} < +\infty
$$
and in this case the fundamental system of norms $\displaystyle \|x\|_k = \sum_n |x_n| a_n^k$ can be replaced by the equivalent system of norms
$$
\|x\|_k =\sup_n |x_n| a_n^k , \quad k \in \mathbb{N}.
$$
    
		The infinite and finite type power series spaces are well known examples of K\"{o}the spaces given by the matrices $(e^{k\alpha_{n}})$ respectively $(e^{-\frac{\alpha_{n}}{k}})$ where $(\alpha_{n})$ is a monotonically increasing sequence going to infinity. The space $A(\mathbb{C})$ of all entire functions on $\mathbb{C}$ and the space $A(\mathbb{D})$ of all holomorphic functions on the unit disc can be represented as an infinite respectively finite type power series spaces.
		
		Smooth sequence spaces were introduced in \cite{Ter69} as a generalization of power series spaces. A K\"{o}the set $A = \left\{(a_n^k)\right\}$ is called a $G_\infty$-set and the corresponding K\"{o}the space $\lambda(A)$ a $G_\infty$-space if $A$ satisfies the followings :
		
   (1) $a_n^1 =1$, $a_n^k \leq a_{n+1}^k$ for each k and n;
	
   (2) $\forall k$ $\exists j$ with $(a_n^k)^2 = O(a_n^j)$

A K\"{o}the set $B = \left\{(b_n^k)\right\}$ is called a $G_1$-set and the corresponding K\"{o}the space $\lambda(B)$ a $G_1$-space if $B$ satisfies the followings :

   (1) $0<b_{n+1}^k \leq b_n^k < 1$ for each $k$ and $n$;
	
   (2) $\forall k$ $\exists j$ with $b_n^k = O((b_n^j)^2)$ 
    
    We need the following result \cite{Cro75}.

\begin{lemma}\label{lem1}
Let $\lambda(A)$ and $\lambda(B)$ be K\"{o}the spaces. A map $T:\lambda(A) \longrightarrow \lambda(B)$ is continuous linear map if and only if for each k there exists m such that
$$
\sup_n \frac{{\|T{e_n} \|}_k}{{\|e_n\|}_m} < + \infty\\
$$
\end{lemma}
If $(a_n)$, $(b_n)$ are two sequences of scalars, then the Cauchy product $(c_n) = (a_n) \ast (b_n)$ of $(a_n)$ and $(b_n)$ is defined by $ \displaystyle{c_n = \sum_{k=1}^n a_{n+1-k} b_k }$.

    Now let $\theta = (\theta_n)$ be a fixed sequence of scalars and let $\lambda(A)$, $\lambda(B)$ be two nuclear $\ell^{1}$-K\"{o}the spaces. We define the Cauchy Product mapping $T_\theta$ from $\lambda(A)$ into $\lambda(B)$ by $T_\theta x = \theta \ast x, \quad x = (x_n) \in \lambda(A).$ So, $T_\theta : \lambda(A) \longrightarrow \lambda(B)$ can be determined by the lower triangular matrix
$$
C =
\begin{pmatrix}
\theta_1 & 0 & 0 & 0 & \cdots \\
\theta_2 & \theta_1 & 0 & 0 & \cdots \\
\theta_3 & \theta_2 & \theta_1 & 0 & \cdots \\
\vdots &  &   & \ddots \\
\end{pmatrix}
$$
\section{Cauchy Product Map on K\"{o}the spaces}
In this section we introduce some necessary and sufficient conditions for the map $T_\theta$ to be linear and continuous. 
\begin{theorem}
Let $\lambda(A)$ be a nuclear K\"{o}the space, $\lambda(B)$ be a nuclear  $G_1$-space. Then the Cauchy product map $T_\theta : \lambda(A) \longrightarrow \lambda(B)$ is linear continuous operator if and only if the following hold:

i) $\theta \in \lambda(B)$

ii) $\lambda(A) \subset \lambda(B)$
\end{theorem}
\begin{proof}
Let $T_\theta : \lambda(A) \longrightarrow \lambda(B)$ be a continuous linear operator.\\
Note that ${\|T_\theta e_n \|}_k = {\|(0,0,...,0,\theta_1,\theta_2,\cdots)\|}_k = \displaystyle{\sup_{j \geq n} |\theta_{j-n+1}| b_j^k} \quad$ for $n\in \mathbb{N}.$
Clearly ${\|e_n\|}_m = a_n^m.$
So, by Lemma \ref{lem1} $\forall k$, $\exists m$, $\exists \rho > 0$ such that
$$
\displaystyle{\sup_{j \geq n} |\theta_{j-n+1}|b_j^k} \leq \rho a_n^m, \quad \forall n \in \mathbb{N}.
$$
Choose $j=n$. Then $\forall k $, $\exists m $, $\exists C>0$ such that
$$
b_n^k \leq C a_n^m                                                      
$$
i.e, $\lambda(A) \subset \lambda(B)$.
Since $T_\theta e_1 \in \lambda(B),$ it follows that $\theta \in \lambda(B)$.

     Conversely, since $B$ is a $G_{1}$-set and by $ii$ and $i$ we have for a given k, there are $m_1(k)$ and $m_2(m_{1})$ such that
\begin{eqnarray*}
\displaystyle \|T_\theta e_n\|_k & = & \sup_{j \geq n } |\theta_{j-n+1}| b_j^k \\
 & \leq &  C_1 \sup_{j \geq n} |\theta_{j-n+1}| (b_j^{m_1})^2 \\
 & \leq & C_1 \sup_{j \geq n} (|\theta_{j-n+1}| b_j^{m_1}) (b_n^{m_1}) \\
 & \leq & C_2 \sup_{j \geq n} (|\theta_{j-n+1}| b_j^{m_1}) (a_n^{m_2}) \\
 & \leq & C_2 \sup_{j \geq n} (|\theta_{j-n+1}| b_{j-n+1}^{m_1}) (a_n^{m_2}) \\
& \leq & C a_n^{m_2}.
\end{eqnarray*}
Therefore, $\forall k$, $\exists m_2$ such that
$$
\sup_n \frac{{\|T_\theta e_n \|}_k}{{\|e_n\|}_{m_2}} < \infty
$$
that is, $T_{\theta}$ is continuous.
\qed
\end{proof}

     We consider the map $T_\theta{'} : \lambda(A) \longrightarrow \lambda(B)$ which is determined by the matrix $C^t$ (the transpose of C) and try to find necessary and sufficient conditions for the continuity of $T_\theta{'}$.
\begin{theorem}
Let $\lambda(A)$ be a nuclear $G_\infty$-space, $\lambda(B)$ be a nuclear K\"{o}the space. Then, $T_\theta{'} : \lambda(A) \longrightarrow \lambda(B)$ which is given above is linear continuous operator if and only if the following hold:

i) $\theta \in \lambda(A)'$

ii) $\lambda(A) \subset \lambda(B)$
\end{theorem}
\begin{proof}
The matrix $C^t$ of the operator $T_{\theta}{'} : \lambda(A) \longrightarrow \lambda(B)$ is the following upper triangular matrix:

$$
C^t =
\begin{pmatrix}
\theta_1 & \theta_2 & \theta_3 & \theta_4 & \cdots \\
0 & \theta_1 & \theta_2 & \theta_3 & \cdots \\
0 & 0 & \theta_1 & \theta_2 & \cdots \\
\vdots &  &   & \ddots \\
\end{pmatrix}
$$
Let $T_{\theta}{'}: \lambda(A) \longrightarrow \lambda(B)$ be a continuous linear operator.\\
Note that $ {\|T_{\theta}{'} e_n\|}_k = {\|(\theta_{n}, \theta_{n-1}, \cdots, \theta_{1},0,0,\cdots)\|}_k = \displaystyle{\sup_{1 \leq i \leq n} |\theta_{n+1-i}| b_i^k} \quad$ for $n \in \mathbb{N}.$
So, by Lemma \ref{lem1} $\forall k$, $\exists m$, $\exists \mu > 0$ such that
$$
\displaystyle{\sup_{1 \leq i \leq n} |\theta_{n+1-i}|b_i^k} \leq \mu a_n^m, \quad \forall n \in \mathbb{N}.
$$
Let $i=1$. Hence $\exists m$, $\exists C = \frac{\mu}{b_{1}^{k}} > 0$ such that

$$
|\theta_{n}| \leq C a_{n}^{m}, \quad \forall n
$$
i.e, $\theta \in \lambda(A)'$.\\
Let $i=n$. Then $\forall k$, $\exists m$ such that

$$
b_{n}^{k} \leq \frac{\mu}{|\theta_1|} a_{n}^{m} 
$$
i.e, $$\lambda(A) \subset \lambda(B)$$
On the other hand, since $A$ is a $G_{\infty}$-set and by $(i)$ and $(ii)$ for a given $k$, there are $m_1$ and $m_2(k)$ and $m=max\left\{m_1,m_2\right\}$ such that
\begin{eqnarray*}
\|T_{\theta}{'} e_n\|_k & = & \sup_{1\leq i \leq n} |\theta_{n-i+1}| b_i^k \\
& \leq & C_1 \sup_{1\leq i \leq n} a_{n-i+1}^{m_1} b_i^k \\
& \leq & C_1 \sup_{1\leq i \leq n} a_{n-i+1}^{m_1} a_i^{m_2} \\
& \leq & C_1 a_n^{m_1} a_n^{m_2} \\
& \leq & C_2 {(a_n^m)}^2.
\end{eqnarray*}
Since $\lambda(A)$ is $G_\infty$ - space, for this $m$, $\exists j$ such that
$$
\sup_n \frac{{(a_n^m)}^2}{a_n^j} < \infty
$$
Therefore, $\forall k $, $\exists j$ such that

$$
\sup_n \frac{{\|T_{\theta}{'} e_n\|}_k}{{\|e_n\|}_j} < \infty
$$
that is, $T_{\theta}{'}$ is continuous.
\qed
\end{proof}

    It is known that $\mathcal{S}$ is a normal sequence space if whenever $\left|x_{i}\right| < \left|y_{i}\right|$ and $y=(y_{i}) \in \mathcal{S}$, then $x=(x_{i}) \in \mathcal{S}$ \cite{Kot69}.
\begin{remark}
Now we write $\theta \in \mathcal{S}$ when the Cauchy product map $T_{\theta}: \lambda(A) \longrightarrow \lambda(B)$ above is continuous. If $\theta,\eta \in \mathcal{S}$, $\lambda \in \mathcal{K}$, then clearly $T_{\theta + \eta}$ and $T_{\lambda \theta}$ will be continuous since $T_{\theta}$ and $T_{\eta}$ are continuous. Hence $\mathcal{S}$ is a vector space.  

    Now, let $\left|\theta_{i}\right| < \left|\eta_{i}\right|, \forall i$, $\eta \in \mathcal{S}.$ Since $T_{\eta}$ is continuous, for all $k$ we find $m$ so that
$$
\sup_{n} \left\{\sup_{j \geq n} \left|\theta_{j-n+1}\right| \frac{b_{j}^{k}}{a_{n}^{m}}\right\} \leq \sup_{n} \left\{\sup_{j \geq n} \left|\eta_{j-n+1}\right| \frac{b_{j}^{k}}{a_{n}^{m}}\right\} < \infty
$$
i.e. $T_{\theta}$ is continuous.

		Therefore $\theta \in \mathcal{S}$. Hence we obtain that $\mathcal{S}$ is a normal sequence space.
\end{remark}


\begin{thebibliography}{20}
\bibitem{Cro75}
	      {\sc L. Crone and W. Robinson},
				{\it Diagonal maps and diameters in {K}\"{o}the spaces},
	      Israel J. Math. {\bf 20}, (1975), 13--22.

\bibitem{Kot69}
        {\sc G. K\"{o}the},
        {\it Topological Vector Spaces 1},
        Springer-Verlag, Berlin, Heidelberg, New York, 1969.
				
\bibitem{Mei97}
        {\sc R. Meise and D. Vogt},
        {\it Introduction to Functional Analysis},
        Clarendon Press, Oxford, 1997.

\bibitem{Pie72}
        {\sc A. Pietsch},
        {\it Nuclear Locally Convex Spaces},
        Springer-Verlag, Berlin-New York, 1972.

\bibitem{Ram79}
	      {\sc M. S. Ramanujan and T. Terzio\u{g}lu},
				{\it Subspaces of smooth sequence spaces},
	      Studia Math. {\bf 65}, (1979), 299--312.
				
\bibitem{Ter69}
	      {\sc T. Terzio\u{g}lu},
				{\it Die diametrale Dimension von lokalkonvexen R\"{a}umen},
	      Collect. Math. {\bf 20}, (1969), 49--99.

\end{thebibliography}
\end{document}